\documentclass[11pt]{article}

\usepackage[T1]{fontenc}
\usepackage{lmodern}

\usepackage{amsmath, amssymb, amsthm}
\usepackage[margin=2.5cm]{geometry}
\usepackage{color, graphicx}
\usepackage{multirow}

\usepackage{etoolbox}
\makeatletter
\patchcmd{\@maketitle}{\LARGE \@title}{\LARGE\bfseries\@title}{}{}

\renewcommand{\@seccntformat}[1]{\csname the#1\endcsname.\quad}
\makeatother

\definecolor{darkblue}{rgb}{0,0,.5}

\usepackage{hyperref}
\hypersetup{
	colorlinks=true,		
	linkcolor=darkblue,		
	citecolor=darkblue,		
	urlcolor=darkblue 		
}

\makeatletter
\def\th@plain{%
	\thm@notefont{}
	\itshape 
}
\def\th@definition{%
	\thm@notefont{}
	\normalfont 
}

\renewenvironment{proof}[1][\proofname]{\par
	\normalfont
	\topsep0\p@\@plus3\p@ \trivlist
	\item[\hskip\labelsep\itshape
	#1\@addpunct{.}]\ignorespaces
}{%
	\qed\endtrivlist
}
\makeatother

\usepackage[capitalize, nameinlink, noabbrev]{cleveref}

\newtheorem{theorem}{Theorem}[section]

\newtheorem{corollary}[theorem]{Corollary}
\newtheorem{proposition}[theorem]{Proposition}
\theoremstyle{definition}
\newtheorem{definition}[theorem]{Definition}
\theoremstyle{definition}
\newtheorem{example}[theorem]{Example}
\theoremstyle{definition}
\newtheorem{remark}[theorem]{Remark}
\theoremstyle{definition}

\theoremstyle{definition}

\numberwithin{equation}{section}

\renewcommand\theenumi{(\roman{enumi})}

\renewcommand{\labelenumi}{\rm (\roman{enumi})}

\usepackage[shortlabels]{enumitem}

\newcommand{\ball}{\mathbb{B}}
\newcommand{\cball}{\overline{\mathbb{B}}}
\newcommand{\R}{\mathbb R}
\def\F{\mathcal{F}}
\def\O{\mathcal{O}}

\newcommand{\argmin}{\ensuremath{\operatorname*{argmin}}}
\newcommand{\inte}{\ensuremath{\operatorname{int}}}
\newcommand{\bd}{\ensuremath{\operatorname{bd}}}
\newcommand{\cl}{\ensuremath{\operatorname{cl}}}
\newcommand{\co}{\ensuremath{\operatorname{co}}}
\newcommand{\dom}{\ensuremath{\operatorname{dom}}}
\newcommand{\epi}{\ensuremath{\operatorname{epi}}}

\newcommand{\loc}{\ensuremath{\operatorname{loc}}}

\begin{document}

\title{\Large Locating Theorems of Differential Inclusions Governed by Maximally Monotone Operators\thanks{Research of M.N.~Dao and M.~Th\'era benefited from the FMJH Program Gaspard Monge for optimization and operations research and their interactions with data science, and was supported by a public grant as part of the Investissement d'avenir project, reference ANR-11-LABX-0056-LMH, LabEx LMH.}}
\author
{Minh N. Dao\thanks{School of Sciences, RMIT University, Melbourne, VIC 3000, Australia. Email: \texttt{minh.dao@rmit.edu.au}.},
~ 
Hassan Saoud\thanks{College of Engineering and Technology, American University of the Middle East, Kuwait. Email: \texttt{hassan.saoud@aum.edu.kw}.},
~~and~ 
Michel Th\'era\thanks{Laboratoire XLIM UMR-CNRS 7252, Universit\'e de Limoges, 87032 Limoges, France and Federation University Australia, Ballarat 3353, Australia. Email: \texttt{michel.thera@unilim.fr, m.thera@federation.edu.au}.}
}

\date{}

\maketitle

\begin{abstract}
In this paper, we are interested in studying the asymptotic behavior of the solutions of differential inclusions governed by maximally monotone operators. In the case where the LaSalle's invariance principle is inconclusive, we provide a refined version of the invariance principle theorem. This result derives from the problem of locating the $\omega$-limit set of a bounded solution of the dynamic. In addition, we propose an extension of LaSalle's invariance principle, which allows us to give a sharper location of the $\omega$-limit set. The provided results are given in terms of nonsmooth Lyapunov pair-type functions.
\end{abstract}

{\small
 
}

\section{Introduction}

In a recent article, by combining elements of two basic paradigms in system dynamics, LaSalle's invariance principle and Lyapunov functions, Dontchev, Krastanov, and Veliov \cite{DKV} established some results on the location of the \emph{$\omega$-limit set} of solutions of a  nonautonomous differential inclusion 
\begin{equation}\label{1}
\dot{x}(t) \in  F (t, x(t)), 
\end{equation}
where $F$ is a set-valued mapping  defined on $\R^n$ and taking  its values in a nonempty subset of $\R^n$. Our main concern in this study is to extend these results and  to provide a localization of the  $\omega$-limit set for an initial value problem governed by a maximally monotone operator.

\subsection{Some background on differential inclusions}

The study of dynamical systems  governed by maximally monotone operators dates back to the first work in the late 1960s and 1970s, done by Komura \cite{komura}, Crandall and Pazy \cite{crandall-pazy69}, Br\'ezis \cite{brezis71}, and later by  many others. Recently, they have received renewed attention because, through adapted discretizations, they make it possible to obtain various numerical algorithms for optimization problems. We invite the readers to the recent survey \cite{csetnek} and on the relations between the continuous and discrete dynamics we refer to \cite{sorin}.
 
Our motivation in this work concerns another aspect of the study of these systems, namely, the localization of the $\omega$-limit set associated with a first-order differential inclusion of the form
\begin{equation} \label{E:eq1} 
\begin{cases}
\dot{x}(t) \in f(x(t)) - A(x(t)) \text{~~a.e.~~} t\in [0,+\infty) \\
x(0) =x_0 \in \cl(\dom A),
\end{cases}
\end{equation}
where $A: \R^n \rightrightarrows \R^n$ is a maximally monotone operator and $f$ is a Lipschitz continuous function defined on $\cl(\dom A) \subseteq \R^n$. The dynamic \eqref{E:eq1} can be seen as a Lipschitz perturbation of the first-order evolution 
\begin{equation}\label{E:eq2}
\dot x(t) \in -A (x(t)) \text{~~a.e.~~} t\in [0,+\infty), 
\quad x(0) =x_0\in \cl(\dom A).
\end{equation}

When $A$ is the subdifferential operator $\partial \varphi$ 
of an extended real-valued lower semicontinuous convex function (a convex potential) $\varphi: \R^n\to\R\cup \{+\infty\}$, several discretizations have been considered for inclusion \eqref{E:eq2}, in order to  construct some algorithms for minimizing $\varphi$. Among them, \emph{subgradient methods} \cite{hiriart-lemarechal-1991,  bgls-2002} and \emph{proximal methods} \cite{martinet} play an important role in numerical optimization. They depend on the choice of the time discretization of the dynamic. Denoting by $x_k$ the $k$th iterate and given a sequence of step sizes $(\gamma_k)$ with  $\gamma_k >0$, we may consider
\begin{enumerate}
\item 
the proximal iteration $x_{k+1}-x_k\in -\gamma_k \partial \varphi(x_{k+1})$ that is equivalent to choosing $x_{k+1} = \argmin_{x\in \R^n} (\varphi(x)+ \frac{1}{2\gamma_k} \|x -x_k\|^2)$. One obtains the \emph{proximal point algorithm} introduced by Martinet \cite{martinet} that converges to $\argmin \varphi$ (supposed to be nonempty) when $\sum_{k=0}^{+\infty} \gamma_k =+\infty$; 
\item
the subgradient iteration $x_{k+1}-x_k\in -\gamma_k \partial \varphi(x_k)$ that subsumes the \emph{gradient descent method} when $\varphi$ is differentiable. 
\end{enumerate}

While the existence and uniqueness of a solution to \eqref{E:eq1}  and \eqref{E:eq2} have been the object of many contributions \cite{Aubin-Cellina, benilan-brezis72, brezis71, Brezis1973, crandall-pazy69}, the ergodic convergence of \eqref{E:eq2} was carried out by Baillon and Br\'ezis \cite{BB76}. In the case where $A$ is the subdifferential operator of a lower semicontinuous convex function $\varphi$, Bruck \cite[Theorem~4]{bruck} proved that whenever $\varphi$ has a minimum, for every initial condition $x_0 \in \cl(\dom \partial \varphi) = \cl(\dom \varphi)$, there exists a unique solution $x(\cdot): [0, +\infty)\to \R^n$, absolutely continuous on $[\delta, +\infty)$ for all $\delta>0$ and which converges to a minimum point of $\varphi$. 

It is important to emphasize two relevant subcases of \eqref{E:eq1}, that is, when $A$ is the subdifferential of a lower semicontinuous convex function $\varphi$ or when $A$ is the normal cone to a closed convex set $C$. This leads us to consider the following two evolution equations:
\begin{equation}\label{E:eq3bis}
\dot x(t) \in f(x(t)) -\partial\varphi((x(t)) \text{~~a.e.~~} t\in [0,+\infty),\quad  x(0) =x_0 \in \cl(\dom \varphi)
\end{equation}
and 
\begin{equation} \label{E:eq3}
\dot x(t) \in f(x(t)) - N_C(x(t)) \text{~~a.e.~~} t\in [0,+\infty),\quad x(0) =x_0 \in C.
\end{equation}

Recently, differential inclusions \eqref{E:eq1}--\eqref{E:eq3}   have attracted much attention since they are relevant in various areas, including, for instance, physics, electrical engineering, economics, biology, population dynamics, and many others. 
For phenomena described by \eqref{E:eq1}, we refer the readers to  \cite{Adlybook, Adly-Goeleven, Adly-Hantoute-Bao, barbu, bastien, brezis71, brogliato, Gonzalez, Morosanu, Showalter}.  
 
One of the  typical  examples described by \eqref{E:eq3bis} is that of \emph{$RLD$ electric circuits} ($R$ is a resistance, $L$ is an inductor, and $D$ is a diode). Indeed, since the diode is a device that constitutes a rectifier which permits the easy flow of charges in one direction and restrains the flow in the opposite direction, the electrical superpotential of the diode is given by $\varphi_D(i) = |i|$, where $i$ stands for the current. Thus, by Kirchoff's law, the dynamic describing the $RLD$ circuit is given by the differential inclusion       
\begin{equation}
\tag{$RLD$}
\frac{di}{dt} + \frac{R}{L} i \in -\partial \varphi_D(i).
\end{equation}
We refer the readers to \cite[Section~3.5]{Adlybook} for further details on electric circuits.

\subsection{Motivation}

As already said, our main concern in this study is the localization of the $\omega$-limit set associated with \eqref{E:eq1} supplied with a given initial condition. Throughout this paper, the $\omega$-limit set is denoted by $\omega(x_0)$. This set is the collection of those points $z\in \R^n$ for which there exists a  solution $x(\cdot;x_0)$ of \eqref{E:eq1}, in a sense given later, defined and bounded on the interval $[0,+\infty)$, and a sequence $(t_k)_{k\in \mathbb{N}}$ with $t_k\in I$ such that $\lim_{k\to +\infty} x(t_k;x_0) =z$.
 
As they are defined, $\omega$-limit sets appear as the sets of points that can be the limit of subtrajectories and they give fundamental information about the asymptotic behavior of dynamical systems. For example, for autonomous dynamical systems, attractors are considered as $\omega$-limit sets. Therefore, $\omega$-limit sets play a crucial role in the study of stability theory and more precisely in  LaSalle's invariance principle that gives a criterion for the asymptotic behavior of autonomous dynamical systems. Moreover, $\omega$-limit sets are nonempty and enjoy noteworthy topological and geometric properties such as compactness, invariance, and connectedness. 
    
They are also considered as the smallest set that a solution approaches. In general, computing $\omega$-limit sets for nonlinear dynamical systems is a difficult task. However, the author in \cite{Hainry} shows that the $\omega$-limit set is not only computable for linear dynamical systems but is also in the case of semi-algebraic sets. 

One of the  main tools for studying the asymptotic behavior of the solution of a dynamical system is LaSalle's invariance principle. It is related to Lyapunov's theory where the positive definiteness of the Lyapunov function is relaxed. However, in some applications, LaSalle's theorem fails to be applicable (see \cite{Arsie2}). Another way to ensure the desired asymptotic stability of the system is to prove that the set where the derivative along the trajectories of the Lyapunov function vanishes is asymptotically stable. Since it is not always possible to find the $\omega$-limit and since this set needs to be known in order to check the convergence of the solution, locating it allows us to give an alternative way to deal with the case where the invariance principle is inconclusive. In fact, Lyapunov functions play an important role in obtaining a set that contains the $\omega$-limit set, easier to find in practice and attracts the solution of the system.

In \cite{Arsie1}, the authors studied the locating problem for autonomous differential equations with a Lipschitz vector field over a Riemannian manifold, using a smooth Lyapunov-type functions. Indeed, they proved that, if the $\omega$-limit set is contained in a closed subset $S$ and if $V$ is a Lyapunov-type function that decreases along the solution on $S$, then it is located in one and only one connected component of the set where the derivative of $V$ along the solution vanishes on $S$. 
In the spirit of \cite{Arsie1}, the authors of \cite{DKV} studied the locating problem for solutions of nonautonomous differential inclusions of the form $\dot{x} \in F(t,x(t))$, where $F$ is a cusco (upper semicontinuous with nonempty compact and convex values) multifunction. Their results are expressed in terms  of a locally Lipschitz Lyapunov pair-type by assuming that the multifunction $F$ is  bounded in a neighborhood of the initial condition. Note that, in both references \cite{Arsie1,DKV}, the function $V$ is neither assumed to be a Lyapunov function nor the set $S$ is assumed to be invariant. If this is the case, the standard LaSalle's invariance principle provides us with the best location of the $\omega$-limit set that is included in $S$.
    
Benefiting from the properties of a maximally monotone operator and using lower semicontinuous Lyapunov-type functions, the first part of this paper is dedicated to the statement of a location theorem. The  result obtained can be viewed as a refined version of the invariance principle. Indeed, since a maximally monotone operator is locally bounded on the interior of its domain (if this is nonempty), the standard assumption in \cite{DKV} is covered. Moreover, using nonsmooth analysis tools and due to the characterization of Lyapunov pairs given in \cite{Adly1} and \cite{BTN}, we are able to give a sufficient condition to ensure the location of $\omega(x_0)$. More precisely, if $(V,W)$ is a lower semicontinuous Lyapunov pair-type for \eqref{E:eq1}, for any closed set $S$ contained in the intersection of the domain of $V$ and the interior of the domain  of the operator $A$, we prove that the $\omega$-limit set is contained in the intersection between $S$ and the set where $W$ is nonpositive. Note that our condition imposed on $V$ and $W$ is more general than the one given in \cite{DKV}. In recent years, extensive research on Lyapunov-pairs for \eqref{E:eq1} has been conducted; see \cite{Adly1,Adly3,BTN}. For example, $(\varphi, \|(\partial\varphi)^\circ\|^2)$ is a Lyapunov pair for the inclusion
\[
\dot x(t) \in -\partial \varphi(x(t)),
\]
where $(\partial \varphi)^\circ$ stands for the minimal norm section of $\partial \varphi$.  

As mentioned previously, LaSalle's invariance principle provides the best location of $\omega(x_0)$ in the case where the set $S$ is  assumed to be invariant. Therefore, in the second part of this paper, we propose a generalized version of the LaSalle invariance principle inspired by the work given in \cite{Haddad}. This allows us to provide a stronger version of the locating problem compared to the one  proposed by the standard invariance principle theorem. Indeed, we prove that, for an invariant set $S$, $\omega(x_0)$ is located in the union of the largest invariant sets contained in intersections over the finite intervals of the closure of the lower semicontinuous Lyapunov level surfaces. In particular, when the Lyapunov function is continuously differentiable, the generalized invariance principle coincides with the standard invariance principle. Our result generalizes the ones given in \cite{Arsie1} which can be covered by taking $A$ equal to the null operator. It also generalizes the results in \cite{DKV}, since it can be applicable to the case where the function $f$ is replaced by a cusco multifunction. In addition, the proposed invariance principle clearly covers the one given in \cite{Qi}, where the operator $A$ is the Fenchel subdifferential of an extended real-valued, proper, lower semicontinuous, and convex function.

\subsection{Content of the paper}

The layout of the paper is as follows. Notations, definitions from nonsmooth analysis, and some properties of maximally monotone operators will be given in the next section. In Section~\ref{sec3}, we will state our main theorem about the location of the $\omega$-limit set followed by some results  discussed under different hypotheses. Finally, Section~\ref{sec4} is devoted to the generalization of LaSalle's invariance principle for the problem \eqref{E:eq1} and its applications.

\section{Preliminaries from convex and variational analysis and monotone operator theory}

We begin this section by providing the notations and gathering some  tools on convex and variational analysis and also on monotone operator theory that we will employ in our subsequent analysis and results.

Our notation is the standard one used in the literature related to these notions \cite{Bauschke, Brezis1973, clarke1, Boris2018, Rock}. Throughout this paper, $\R^n$ is the $n$ dimensional Euclidean space with inner product $\langle\cdot, \cdot\rangle$ and induced norm $\| \cdot\|$, i.e., for all $x \in \R^n$, $\| x\| := \sqrt{\langle x, x\rangle}$. We denote by $\ball(x;r)$ the open ball in $\R^n$ with center $x$ and radius $r$. For a set $S \subseteq \R^n$, $\inte S$, $\bd S$, $\cl S$, and $\co S$ stand for the interior, the boundary, the closure, and the convex hull of $S$, respectively. The \emph{distance function} to the set $S$ is defined by $d(x,S) := \inf\{\| x-y \|:\, y\in S\}$ and the \emph{projection} onto $S$ is defined by $P_S(x) =\{y\in S:\, \|x-y\| =d(x,S)\}$. We denote the closed ball around the set $S$ with radius $r$ by $\cball(S;r) :=\{x:\, d(x,S) \leq r\}$. In addition, $S^\circ$ stands for the set of points of \emph{minimal norm} in $S$, i.e., $S^\circ := \{x\in S:\, \forall s\in S,\; \|x\| \leq \|s\|\} = P_S(0)$.
   
Recall that $\mathbf{L}^\infty ([a,b]; \R^n)$ is the Banach space of all the (equivalence classes by the relation equal almost everywhere) measurable functions $f:[a,b]\to \R^n$ that are essentially bounded on $[a,b]$. It is equipped with the norm $\| f\|_\infty =\operatorname{ess sup}_{x \in [a,b]} \| f(x)\|$. $\mathbf{L}^\infty_{\loc}([a,b];\R^n)$ refers to the space of those functions $f$ such that for every compact $K\subseteq [a,b]$, $f\in \mathbf{L}^\infty (K;\R^n)$.

Let $\varphi: \R^n \to \R\cup\{+\infty\}$ be an extended real-valued function. The \emph{(effective) domain} and \emph{epigraph} of $\varphi$ are defined by  
\[
\dom \varphi :=\{x\in \R^n:\, \varphi(x)<+\infty\}
\text{~~and~~}
\epi\varphi:=\{(x,\alpha)\in \R^n\times\mathbb{R}:\, \varphi(x)\leq\alpha\}.
\]
We say that $\varphi$ is \emph{proper} if $\dom \varphi\neq \varnothing$ and that $\varphi$ is \emph{convex} if $\epi\varphi$ is convex. The function $\varphi$ is said to be \emph{lower semicontinuous} at $y \in \R^n$ if for every $\alpha \in \R$ with $\varphi(y) > \alpha$, there is a $\delta > 0$ such that 
\[
\forall x\in \ball(y;\delta),\quad \varphi(x) > \alpha.
\] 
We simply say that $\varphi$ is lower semicontinuous if it is lower semicontinuous at every point of $\R^n$. Equivalently, $\varphi$ is lower semicontinuous if and only if its epigraph is closed.

Given a subset $S$ of $\R^n$ and $\alpha \in \R$, the set $[ \varphi = \alpha]_{\mid S} := \{x\in S:\,  \varphi (x) = \alpha \}$ stands for the \emph{$\alpha$-level set} of the function $\varphi$. For $\alpha, \beta \in \R$,  $\alpha < \beta$, the set 
\[
[\alpha \leq  \varphi \leq \beta]_{\mid S} :=\{x\in S:\, \alpha \leq   \varphi(x) \leq \beta \}
\]
is called the \emph{$[\alpha,\beta]$-sublevel set} of the function $\varphi$.  

We denote by $\mathcal{F}(\R^n)$ (resp., $\F^+(\R^n)$) the set of extended real-valued, proper, and lower semicontinuous functions (resp., nonnegative). Finally, given a subset $S$ of $\R^n$ and a function $\varphi \in \F^+(\R^n)$, we note $S^+_\varphi  := \{x\in S :\, \varphi(x) > 0\}$. Observe that, by the lower semicontinuity of $\varphi$, the set $S^+_\varphi $ is open in $S$.

We proceed by giving some definitions and results from \emph{nonsmooth analysis}. The basic references for these notions and facts can be found in details in \cite{clarke1,clarke2,Rock}. Let $\varphi$ be a function of $\F(\R^n)$ and let $x\in \dom \varphi$. We say that a vector $\zeta \in \R^n$ is a \emph{proximal subgradient} of $\varphi$ at $x$ if there exist $\eta > 0$ and $\sigma \geq 0$ such that 
\[
\forall y \in \ball(x; \eta),\quad \varphi(y) \geq \varphi(x) + \langle \zeta \,, y-x \rangle - \sigma \| y-x\|^2.
\]
The \emph{proximal subdifferential} of $\varphi$ at $x$ is the collection of all proximal subgradients and is denoted by  $\partial_P \varphi (x)$. The set $\partial_P \varphi (x)$ is convex, possibly empty, and not necessarily closed.
 
A vector $\zeta \in \R^n$ is called a \emph{Fr\'echet subgradient} of $\varphi$ at $x$ if the following inequality holds
\[
\forall y \in \R^n,\quad \varphi(y) \geq \varphi(x) + \langle\zeta, y - x\rangle + o(\| y - x\|).
\]
The set of such $\zeta$ is called the \emph{Fr\'echet subdifferential} of $\varphi$ at $x$, and it is denoted by $\partial_F \varphi(x)$.

The \emph{limiting subdifferential} of $\varphi$ at $x$, denoted by $\partial_{L}\varphi(x)$, is the set of vectors $\xi\in \R^n$ such that there exist a sequence $(x_{k})_{k\in \mathbb{N}}$ with $x_{k}\stackrel{\varphi} \to x$ and a sequence $(\xi_{k})_{k\in \mathbb{N}}$ with $\xi_{k}\in\partial_{P}\varphi(x_{k})$ and $\xi_{k} \to\xi$. Here, the notation $x_k \stackrel{\varphi}{\to} x$ means that
 $x_k \to x$ and $\varphi(x_k) \to \varphi(x)$.

A vector $\zeta \in \R^n$ is called a \emph{horizon subgradient} of $\varphi$ at $x$ if there exist sequences $(\alpha_k)_{k\in \mathbb{N}} \subset \R^+$ and $(x_k)_{k\in \mathbb{N}}, (\zeta_k)_{k\in \mathbb{N}} \subset \R^n$ such that
\[
\alpha_k \downarrow 0,\,\,x_k \stackrel{\varphi}\to x,\,\,
\zeta_k \in \partial_P \varphi(x_k),\,\,\alpha_k \zeta_k \to \zeta.
\]
The set of such $\zeta$ is called the \emph{horizon subdifferential}  of $\varphi$ at $x$ and is denoted by $\partial_{\infty} \varphi(x)$.

Finally, the \emph{Clarke subdifferential} of $\varphi$ at $x$ is given by 
\[
\partial_C \varphi(x) = \cl\Big(\co\big(\partial_L \varphi(x) + \partial_{\infty} \varphi (x)\big)\Big).
\]
It is clear, from the definitions above, that $\partial_P \varphi (x) \subseteq \partial_F \varphi (x)\subseteq \partial_L \varphi (x) \subseteq \partial_C \varphi (x)$. In addition, if the function $\varphi$ is $C^1$ near $x$, then $\partial_P \varphi (x) \subseteq \{\varphi'(x)\} = \partial_C \varphi(x)$. If $\varphi \in C^2$, then $\partial_P \varphi (x) = \partial_C \varphi (x) =  \{\varphi'(x) \}$. In the case where $\varphi$ is Lipschitz near $x$, then $\partial_{\infty} \varphi (x) =\{0\}$ and $\partial_C \varphi(x) = \cl\big(\co (\partial_L \varphi(x))\big)$. 
For $x \notin \dom \varphi$, we have $\partial_P \varphi(x) = \partial_F \varphi(x) = \varnothing$. If the function $\varphi$ is convex, then for every $x \in \R^n$, we have 
$ \partial_P \varphi (x) =  \partial_F \varphi (x) = \partial_L \varphi (x) = \partial_C \varphi (x) = \partial \varphi (x)$, where $\partial \varphi(x)$ stands for the \emph{Fenchel subdifferential} of $\varphi$ at $x$ which is defined by  
\[
\zeta \in \partial \varphi(x) \iff \forall y \in \R^n,\quad \varphi(y) \geq 
 \varphi(x) +\langle\zeta, y-x\rangle.
\]

Let $S$ be a nonempty subset of $\R^n$. The \emph{indicator function} of $S$ is the function $\iota_S$ taking the values $0$ on $S$ and $+\infty$ off $S$. The \emph{proximal}, \emph{Fr\'echet}, \emph{limiting}, and \emph{Clarke} \emph{normal cones} to $S$ are defined as 
\[
N_S^{\bullet}(x) := \partial_\bullet \iota_S(x),
\]
where $\bullet$ stands for $P$, $F$, $L$, or $C$. A geometric characterization of the notion of subdifferentials, previously defined, is given by
\[
\zeta \in \partial_{\bullet} \varphi (x) \iff (\zeta,-1)\in
 N_{\epi \varphi}^{\bullet}(x, \varphi(x)).
\]

Another central tool concerns the theory of maximally monotone operators.

\subsection{Maximally monotone operators}

A multifunction $A: \R^n \rightrightarrows \R^n$ is said to be \emph{monotone} if 
\[
\forall (y_1,y_2) \in A(x_1) \times A(x_2),\quad \langle y_1 - y_2,x_1 - x_2 \rangle \geq 0. 
\]
The \emph{domain} of $A$ is the set 
\[
\dom A = \{x\in \R^n :\,  A(x) \neq \varnothing \}.
\]
A monotone operator $A$ is \emph{maximally monotone} provided its graph $ \{(x,y)\;: y\in A(x)\}$ cannot be properly enlarged without destroying monotonicity. 

Unlike its closure, the domain of a maximally monotone operator is not necessarily closed and convex (it is nearly convex; e.g., see \cite{Rock}). However, its values are closed and convex but they are not supposed to be bounded or even nonempty. A typical example of a maximally monotone operator is the Fenchel subdifferential of an extended real-valued proper lower semicontinuous convex function $\varphi$. We have 
\[
\dom (\partial \varphi) \subseteq \dom \varphi \subseteq \cl(\dom \varphi) = \cl(\dom \partial \varphi).
\]
	
\begin{definition}
Let $S$ be a subset of $\R^n$. A maximally monotone operator $A: \R^n \rightrightarrows \R^n$ is said to be
\begin{enumerate}
\item  
\emph{locally minimally bounded} on $S$ if for all $x\in S \cap \dom A$, there exist $M, r > 0$ such that 
\[
\forall y \in S\cap \dom A \cap \ball(x;r), \quad \| A^\circ (y)\| \leq M;
\] 
\item  
\emph{locally bounded} on $S$ if for all $x\in S \cap \dom A$, there exist $M, r > 0$ such that 
\[			
\forall y\in S\cap \dom A \cap \ball(x;r),\ \forall u\in A(y),\quad \| u \| \leq M.
\]
\end{enumerate}
\end{definition}
It is clear that a locally bounded operator is locally minimally bounded.

\begin{proposition}[\cite{Rock1, phelps}]
\label{proposition:prop1}
Let $A: \R^n \rightrightarrows \R^n$ be a maximally monotone operator and let $x\in \cl(\dom A)$. Then the following hold:  
\begin{enumerate}
\item 
$A$ is locally bounded at $x$ if and only if $x\in \inte(\dom A)$.
\item
If $\inte(\co (\dom A)) \neq \varnothing$, then $\inte(\dom A) $ is a nonempty convex set and also $\inte (\dom A) =\inte (\co(\dom A)) = \inte (\cl (\dom A))$. 
\end{enumerate}
\end{proposition}

As an immediate consequence of \cref{proposition:prop1}, the Fenchel subdifferential of a proper lower semicontinuous convex function is locally bounded on the interior of its domain. When applied to the indicator function of a convex closed set $C$, this subsumes that the normal cone operator $N_C$ to a closed convex set $C$ is locally bounded on $\inte C$.
 
Given a maximally monotone operator $A$ and a Lipschitz continuous function $f$ defined on $\cl (\dom A) \subseteq \R^n$, we consider again \eqref{E:eq1}. For a fixed $T> 0$ and $x_0 \in \cl (\dom A)$, it is known that there exists a unique absolutely continuous function $x(\cdot;x_0): [0,T] \to \R^n$ with $\dot{x}(\cdot;x_0) \in  \mathbf{L}^\infty_{\loc} ((0,T], \R^n)$ and, for all $t >0$, $x(t;x_0) \in \dom A$ such that $x(\cdot;x_0)$ satisfies \eqref{E:eq1}.

Note that the existence of such a solution occurs if $x_0 \in \dom A$, $\inte (\co (\dom A)) \neq \varnothing$, and the underlying space is finite dimensional (which is the case here), or if $A= \partial \varphi$, where $ \varphi$ is an extended real-valued proper lower semicontinuous convex function. 

Furthermore, we have 
\[
\dot{x}(\cdot;x_0) \in \mathbf{L}^\infty ([0,T],\R^n) \iff x_0 \in \dom A.
\]
In this case, $x(\cdot;x_0)$ is right differentiable at each 
$s\in [0,T)$ and 
\[
\frac{d^+ x(\cdot;x_0)}{dt}(s) = f(x(s;x_0)) - P_{A(x(s))}(f(x(s))) = (f(x(s)) - A(x(s)))^\circ.
\]
Also, we have the semi-group property 
\[
\forall s,t \geq 0,\quad x(s;x(t;x_0)) = x(s+t;x_0).
\]

Now, let us recall some important notions concerning the dynamic \eqref{E:eq1}.

\begin{definition}
The \emph{$\omega$-limit set} associated with \eqref{E:eq1}, supplied with the initial condition $x(0) = x_0$, is defined by 
\begin{equation*}
\omega (x_0) :=  \bigcap_{T>0}\cl\left(x([T,+
\infty)\right).
\end{equation*}
\end{definition}
In other terms, a point $z\in \omega (x_0)$ if there exists a sequence $(t_k)_{k\in \mathbb{N}}$ with $t_k \to +\infty$ as $k \to +\infty,$ such that $\lim_{k \to +\infty} x(t_k;x_0) = z$. It is well known that for every bounded solution of \eqref{E:eq1}, the $\omega$-limit set is nonempty, compact, invariant, and connected (see \cite{khalil, Haddad}). Moreover, it is easy to check that if $\omega (x_0)$ contains a Lyapunov stable equilibrium point $z$, then $\lim_{t \to +\infty} x(t;x_0) = z$ and $\omega (x_0) = \{z\}$.

\begin{definition}
\label{def3} 
A set $S \subseteq \cl (\dom A)$ is said to be an \emph{invariant set} with respect to \eqref{E:eq1} if for all $x_0 \in S$ and all $t \geq 0$, $x(t;x_0) \in S$.
\end{definition}

Finally, we consider the ordinary differential equation of the form 
\begin{equation}
\label{eq2}
\dot{x}(t) = f(x(t)),\quad x(0) = x_0 \in \mathbb{R}^n,  
\end{equation}
where $f$ is a locally Lipschitz function defined on an open set $\O$ of $\R^n$. This is a special case of \eqref{E:eq1} when $A \equiv 0$. In the setting of \eqref{eq2}, we recall the well-known LaSalle's invariance principle for its importance since it is the core of the stability theory of dynamical systems. 

\begin{theorem}[\cite{khalil}]\label{theorem:thm0}
Let $S \subseteq \O$ be a compact invariant set with respect to \eqref{eq2}. Let $V: \O \to \R$ be a continuously differentiable function such that $\dot{V}(x) \leq 0$ for all $x\in S$. If $E:= \{ x\in S:\,  \dot{V}(x) = 0 \}$, then every solution of \eqref{eq2} starting in $S$ approaches the largest invariant set $M$ contained in $E$:  
\[
\lim_{t\to +\infty} d(x(t), M) =0.
\]
In particular,
\[
\omega (x_0) \subseteq M \subseteq E \subseteq S. 
\]
\end{theorem}
We denote by $\dot{V}(x(t;x_0)) := \frac{dV}{dt} (x(t;x_0))$ the derivative of $V$ along the solution of \eqref{eq2}: $\dot{V}(x(t;x_0)) =\langle\nabla V(x),f(x) \rangle$.

\section{Location via closed sets}
\label{sec3}

As mentioned previously in \cref{theorem:thm0}, the $\omega$-limit set is contained in the largest invariant set $S$ where the derivative of $V$ along the solution vanishes for all $x \in S$. The main result of this section provides a location theorem of the $\omega$-limit set of bounded solutions of \eqref{E:eq1} via sets which are not necessarily invariant. This result is helpful when the invariance theorem cannot be applied directly. Thus, it is considered as a refined version of the invariance principle. In addition, using the notion of a lower semicontinuous Lyapunov pair-type, our result can also be seen as a generalization of results given in \cite{Arsie1} and \cite{DKV}.

\vskip 2mm
\textbf{Blanket assumptions.}  Throughout this section we assume that $V,W \in \F(\R^n)$ are such that the following hold:
\begin{enumerate}
\renewcommand{\theenumi}{(A\arabic{enumi})}
\renewcommand{\labelenumi}{(A\arabic{enumi})}
\item\label{A1} 
For all $x_0 \in \cl(\dom A)$ and all $\rho_0 > 0$, there exists $\bar y \in \ball(x_0;\rho_0) \cap \dom A$ such that $\ball(\bar y;\rho_{\bar y}) \cap \dom V  \subseteq \inte (\dom A)$ for some $\rho_{\bar y }> 0$.
\item\label{A2} 
There exists a closed subset $S$ of $\R^n$ such that $\omega(x_0) \subseteq S \subseteq \dom V \cap \ball(\bar y;\rho_{\bar y})$. 
\item
$\dom W= \left( \dom V \cap \ball(\bar{y}; \rho_{\bar y}) \setminus S\right)\cup S^+_W$.
\end{enumerate}

\begin{remark}
The blanket assumption~\ref{A1} looks somewhat technical but it is very useful in the rest of this paper. In fact, it determines the relationship between the domain of the operator $A$ and where the function $V$ is defined. On the other hand, assumption~\ref{A2} is equivalent to saying that every bounded solution of \eqref{E:eq1} starting at $x_0$ is attracted by $S$. Combined together, assumptions~\ref{A1} and \ref{A2} ensure the nonemptiness of $\inte (\dom A)$ and thus, due to \cref{proposition:prop1}, we have the local boundedness of the operator $A$ at each point of this set.
\end{remark}

Now, we are ready to state the main theorem of this section.
\begin{theorem}
\label{theorem:thm1}
Given $x_0 \in \cl(\dom A)$, suppose that the corresponding solution $x(\cdot;x_0)$ of \eqref{E:eq1} is bounded and that
\begin{enumerate}
\renewcommand{\theenumi}{($\mathcal{H}_\arabic{enumi}$)}
\renewcommand{\labelenumi}{($\mathcal{H}_\arabic{enumi}$)}
\item\label{H1} 
for all $x\in \dom W$, 
\begin{equation}\label{eq3}
\sup_{\zeta \in \partial_{\bullet} V(x)}\inf_{v \in A(x)}\langle \zeta,f(x) -v \rangle \leq -W(x),
\end{equation}
where $\partial_{\bullet}$ stands for $\partial_P$ or $\partial_F$;

\item\label{H2} 
for all $x\in \dom V$,      
\begin{equation*}
V(x) = \liminf_{w \xrightarrow{\dom A} x} V(w);
\end{equation*}

\item\label{H3} 
the set $V(S^+_W)\setminus V(S\setminus S^+_W)$ is dense in $V(S^+_W)$.
\end{enumerate}
Then $\omega (x_0) \subseteq S\setminus S^+_W$.
\end{theorem}

\begin{remark}
Condition~\ref{H3} seems to be a little technical but it is essential for the proof of the theorem. Moreover, it gives information concerning the dynamic outside $S$ according to what happens inside $S$. 
\end{remark}

\begin{proof}[Proof of \cref{theorem:thm1}]
The proof is inspired by the one given by Dontchev, Krastanov, and Veliov \cite{DKV} and will be given by contradiction in several steps. First of all,  observe that the set $\dom W= ( \dom V \cap \ball(\bar y;\rho_{\bar y}) \setminus S)\cup S^+_W$ is open relative to $\dom V$.
	
\emph{Step 1.} By the lower semicontinuity of $W$, for each $x \in S^+_W$, there exists $r_x > 0$ such that $\ball(x;r_x) \subseteq ( \dom V \cap \ball(y;\rho_y) \setminus S)\cup S^+_W$ and where for all $y \in \ball(x;r_x)$ we have $y \in (\ball(x;r_x))^+_W$. 

Set
\begin{equation*}
\mathcal{O} := \bigcup_{x\in S^+_W} \ball(x;r_x).
\end{equation*}
By definition, $\O$ is an open set containing $S^+_W$ such that $\O \subseteq \dom W$. Suppose that 
\begin{equation}
\label{eq4}
\omega(x_0) \subseteq S\setminus S^+_W	
\end{equation} 
fails. Since $\omega(x_0) \subseteq S$, we may pick some  $\bar z \in \omega(x_0) \cap S^+_W$. By the definition of $\omega(x_0)$, there exists a bounded solution $x(\cdot;x_0)$ of \eqref{E:eq1} and a sequence $(t_k)_{k\in \mathbb{N}}$ such that $\lim_{k\to +\infty} t_k = +\infty$ and 
\begin{equation}
\label{eq5}
\lim_{k \to +\infty}x(t_k;x_0) = \bar z. 
\end{equation}
Since $x(t;x_0)$ is bounded for $t\geq 0$, then there exists a compact $C \subseteq \R^n$ such that $x(t;x_0) \in C$ for all $t\geq 0$. As $W(\bar z)>0$, we can take $\lambda$ such that $0 <\lambda <\min\{1, W(\bar z)\}$. Since $W$ is lower semicontinuous at $\bar z$, there exists $\rho>0$ such that 
\begin{equation}\label{eq:Wx}
\forall x\in \cl \ball(\bar z; \rho), \quad W(x) >\lambda.
\end{equation}
By shrinking $\rho$ if necessary, we have that $\ball(\bar z; \rho) \subseteq \O$.

\emph{Step 2.} We claim that for all $\varepsilon > 0$, there exist $c\in (V(\bar z)-\varepsilon,V(\bar z) + \varepsilon)$ and $\delta >0$ such that 
\[
[V = c]_{\mid \cball(S;\delta)\cap C} \subseteq \O. 
\]
Indeed, suppose on the contrary that for each $c\in  (V(\bar 
z)-\varepsilon,V(\bar z) + \varepsilon)$ and for every $\delta = \frac{1}{n}$ with $n\in \mathbb{N}\setminus \{0\}$, there exists $x_n \in \cball(S;\frac{1}{n})\cap C$ such that $V(x_n) = c$ and $x_n \notin \O$. By the compactness of $C$, and relabeling if necessary, we may assume that the sequence $(x_n)$ converges to $\bar{x} \in S\cap C$ with $\bar{x} \notin \O$ and therefore $\bar{x} \notin S^+_W$. In addition, according to~\ref{H2}, we have $\liminf_{x_n \to \bar{x}} V(x_n) = V(\bar{x}) = c$. This yields $c = V(\bar{x}) \in V(S\setminus S^+_W)$. Since this fact holds for every $c \in (V(\bar z)-\varepsilon,V(\bar z) + \varepsilon)$, we deduce that, for all $\varepsilon>0$, 
\begin{equation}\label{sissinette} 
(V(\bar z)-\varepsilon,V(\bar z) + \varepsilon) \subseteq V(S\setminus S^+_W) \;\text{and }\; V(\bar z) \in V(S^+_W),
\end{equation} 
which is a contradiction with assumption~\ref{H3}.

\emph{Step 3.} Let $\tau >0$ be such that $\tau\|\dot{x}\|_\infty \leq \frac{\rho}{2}$. Set $\varepsilon :=\frac{\tau\lambda}{2} >0$. By \emph{Step~2}, we find $c\in (V(\bar z)-\varepsilon,V(\bar z) + \varepsilon)$ and $\delta\in (0,\frac{\rho}{2})$ such that 
\begin{equation}\label{eq:V=c}
[V = c]_{\mid \cball(S;\delta)\cap C} \subseteq \O.
\end{equation}
By~\ref{A2}, there exists $T >0$ such that $x(t;x_0) \in \cball(S;\delta)$ for every $t\geq T$. Using \eqref{eq5} and \ref{H2}, there exists an index $i$ such that $t_{i} > T$ with $x(t_{i};x_0) \in \ball(\bar z;\delta)$ and $V(x(t_i;x_0)) -V(\bar z) <\varepsilon$. Then, for all $t\in [t_{i},t_{i}+\tau]$,
\begin{align}\label{eq6} 
\| x(t;x_0) - \bar z\|  & =  \left\| x(t_i;x_0) + \int_{t_i}^t \dot{x}(s;x_0)~ds - \bar z\right\| \\
& \leq \| x(t_i;x_0) - \bar z\| + \int_{t_i}^t \| \dot{x}(s;x_0) \|~ds \nonumber\\ 
& \leq \delta + (t - t_i) \| \dot{x} \|^{\loc}_\infty \nonumber\\
& \leq \delta + (t - t_i) \| \dot{x} \|_\infty \nonumber\\
& \leq \delta + \tau \| \dot{x} \|_\infty \nonumber\\ 
& \leq \frac{\rho}{2} +\frac{\rho}{2} =\rho. \nonumber
\end{align}
Therefore, 
\begin{equation}\label{eq:xt in ball}
\forall t\in [t_{i},t_{i}+\tau],\quad x(t;x_0) \in \ball(\bar 
z;\rho).
\end{equation}

\emph{Step 4.} Now, combining assumptions~\ref{H1}, \ref{H2} ($x(t_i;x_0)  \xrightarrow{V} x$), and the fact that the operator $A$ is locally bounded with respect to the subspace topology on  $\dom V$, we may apply \cite[Theorem~3.1]{Adly1} (or \cite[Corollary~3.14]{BTN}) to obtain that, for $t\geq 
t_i$, 
\begin{equation}\label{eq:Vx}
V(x(t;x_0)) \leq V(x(t_i;x_0)) - \int_{t_i}^{t} W(x(s;x_0))~ds.
\end{equation}
Together with \eqref{eq:Wx} and \eqref{eq:xt in ball}, we have that, for all $t\in [t_{i},t_{i}+\tau]$,
\begin{equation}\label{eq7}
V(x(t;x_0)) \leq V(x(t_i;x_0)) - (t-t_i)\lambda < V(\bar z) + 
\varepsilon - (t-t_i)\lambda.
\end{equation}
By letting $t =t_i +\tau$ and recalling $\varepsilon =\frac{\tau\lambda}{2}$,
\begin{equation*}
V(x(t_i+\tau;x_0)) < V(\bar z) + \varepsilon - \tau\lambda < c + 2\varepsilon - \tau\lambda = c.
\end{equation*}

Next, we show that
\begin{equation}\label{eq9}
\forall t\geq t_i+\tau,\quad V(x(t;x_0)) < c.
\end{equation}
Assume on the contrary the existence of some   $s > t_i +\tau$ 
such that 
\begin{equation}\label{eq8}
V(x(s;x_0)) = c \text{~~and~~} \forall t\in [t_i +\tau,s),\quad V(x(t;x_0)) < c.
\end{equation}
Since $x(t;x_0) \in C$ for all $t \geq 0$ and $x(t;x_0) \in \cball(S; \delta)$ for all $t \geq T$, it holds that 
$x(s;x_0) \in \cball(S; \delta) \cap C$. By combining this with \eqref{eq:V=c} and the equality in \eqref{eq8}, $x(s;x_0) \in \mathcal{O}$. Since $x(\cdot; x_0)$ is continuous and $ \mathcal{O}$  is open, there exists $d > 0$ such that $ s - d > t_i + \tau$ and $x(t;x_0) \in \mathcal{O}$ for all $t \in [s-d, s]$. It then follows from the definition of $\mathcal{O}$ that $W(x(t;x_0)) > 0 $ for all $t \in [s-d, s]$. According to \eqref{eq:Vx},
\begin{equation*}
V(x(s;x_0)) \leq V(x(s-d;x_0)) - \int_{s-d}^s W(x(t;x_0))~dt \leq V(x(s-d;x_0)) < c,
\end{equation*}
which contradicts the equality in \eqref{eq8}. Thus, we get \eqref{eq9}.

\emph{Step 5.} By~\eqref{eq5} and the lower semicontinuity of $V$, there exists $j$ such that $t_j > t_i + 2\tau$, $V(x(t_j;x_0)) > V(z) - \varepsilon$, and $\| x(t_j;x_0) - z\| <\delta$. Similarly to \eqref{eq6}, for each $t \in [t_j-\tau,t_j]$, we have 
\begin{align*}
\| x(t;x_0) - \bar z\|  & =  \left\| x(t_j;x_0) - \int_{t}^{t_j} \dot{x}(s;x_0)~ds - \bar z\right\| \\
& \leq \| x(t_i;x_0) - \bar z\| + \int_{t}^{t_j} \| \dot{x}(s;x_0) \|~ds \\
& \leq \delta + (t_j - t) \| \dot{x} \|^{\loc}_\infty \\
& \leq \delta + (t_j - t) \| \dot{x} \|_\infty < \rho.
\end{align*}
Thus, $x(t;x_0) \in \ball(\bar z;\rho)$ for all $t \in [t_j -\tau,t_j]$. Similarly to \eqref{eq7}, we obtain that 
\begin{align*}
V(x(t_j;x_0)) & \leq V(x(t_j-\tau;x_0)) - \int_{t_j - \tau}^{t_j} W(x(t;x_0))~dt \\ 
& \leq V(x(t_j - \tau;x_0)) - \tau\lambda. 
\end{align*}
From the latter inequality, we deduce that 
\[
V(x(t_j-\tau;x_0)) \geq V(x(t_j;x_0) + 2\varepsilon \geq (V(\bar z) - \varepsilon) + 2\varepsilon \geq c,
\]
which contradicts \eqref{eq9}, since $t_j - \tau > t_i +\tau$. The contradiction obtained is a consequence of the assumption that \eqref{eq4} is not true. 
\end{proof}

\begin{example}
\label{example:ex1}
Consider the case where $A \equiv \partial  \varphi$ for $\varphi\in\F(\R^n)$ and convex. Let $x_0 \in \dom (\partial \varphi)$ and suppose that the corresponding solution $x(\cdot;x_0)$ of \eqref{E:eq1} is bounded. If we suppose that the assumptions~\ref{H2} and \ref{H3} are satisfied and that for all $x\in ((\dom V \cap \ball(\bar y;\rho_{\bar y}))\setminus S) \cup S^+_W$, 
\[
\sup_{\zeta \in \partial_{\bullet} V(x)}\inf_{v \in \partial \varphi(x)}\langle \zeta,f(x) -v \rangle \leq -W(x),
\]
then \cref{theorem:thm1} shows that $\omega(x_0)\subseteq S\setminus S^+_W$.
\end{example}

\begin{remark}
\label{remark:rem1}
The result of \cref{theorem:thm1} holds under different assumptions. Thus, a series of specific results can be derived. For instance,
\begin{enumerate}
\item 
assumption~\ref{H1} can be replaced (see \cite{Adly1}) by the condition that for all $x\in ((\dom V \cap \ball(\bar y; \rho_{\bar y}) )\setminus S)\cup S^+_W$, we have  
\[
\inf_{v \in A(x)} V'(x,f(x)-v) \leq W(x),
\]
where $V'(x,f(x)-v)$ is the \emph{contingent directional derivative} of $V$ at $x\in \dom V$ in the direction $v$ and given by  
\[
V'(x,v) = \liminf_{t\to 0^+,w\to v} \frac{f(x+tw) - f(x)}{t}; 
\]	
\item 
an interesting and important case is when the function $V$ is also convex, defined on $\dom V \cap \dom A$ and where $\dom V$ is open.  In this case, $V$ becomes locally Lipschitz on the interior of its domain (in fact, on its domain). Then, combining the proof of \cref{theorem:thm1} and the proof in \cite{DKV}, we are able to prove the result which can be seen, in some way, as a particular case of the result given in \cite{DKV}.
	  
Moreover, if $\dom V$ is contained in $\dom A$, then 
assumption~\ref{H2} is naturally satisfied and assumption~\ref{H1} becomes for every $x\in (\dom V\setminus S) \cup S^+_W,$ we have  
\begin{equation*}
\sup_{\zeta \in \partial_{\bullet} V(x)} \langle\zeta,(f(x)-
Ax)^\circ \rangle\leq -W(x),
\end{equation*}
where $(f(x)-Ax)^\circ = P_{f(x) - A(x)}(0)  =  f(x) - P_{A(x)}(f(x))$ is the element of minimal norm in $f(x) - A(x)$;
\item  
observe that, the lower semicontinuity of $W$ can be replaced by  its Lipschitz continuity, since every lower semicontinuous function can be regularized by a sequence of Lipschitz functions on every bounded subset of $\R^n$ (see, e.g., \cite{clarke2}). Thus, for $W \in \F^+(\R^n)$, using the well-known quadratic Inf-convolution, there exists a sequence of Lipschitz functions $(W_k)$ that converges pointwise to $W$ and  such that for each $k$ and each $y \in \R^n$, we have that
\[
W_k (y) >0 \text{~~if and only if~~} W(y) >0.
\]
\end{enumerate}
\end{remark}

\begin{remark}
\label{remark:rem0}
Consider the differential inclusion
\begin{equation}\label{cusco}
\dot{x}(t) \in F(x(t)) - A(x(t)) \text{~~a.e.~~} t\in [0,+\infty),\quad x(0) = x_0 \in \cl(\dom A), 
\end{equation}
where $F: \R^n \rightrightarrows \R^n$ is a multifunction and $A$ is a maximally monotone operator. We distinguish here two different cases:
\begin{enumerate}
\item 
If $F$ is of type cusco and if $A \equiv 0$, then \eqref{cusco} is reduced to the standard differential inclusion. In addition, if $F$ is bounded in a neighborhood of the initial condition, then the results in \cite{DKV} are covered by taking $V$ and $W$ lower semicontinuous and not only locally Lipschitz.   
\item
If $F$ is Lipschitz cusco, then thanks to a selection theorem given in \cite{Artstein}, Adly, Hantoute, and Nguyen \cite{Adly-Hantoute-Bao} rewrote \eqref{cusco} as \eqref{E:eq1} and proved that \eqref{cusco} has at least one solution.
\end{enumerate}
\end{remark}

An immediate special case of \eqref{E:eq1} is when $A$ is the null operator, and thus, problem \eqref{E:eq1} reduces to \eqref{eq2}. In the following corollary, we provide a generalized version of the result given in \cite{Arsie1}, not only when the used Lyapunov function is continuously differentiable but also when it is lower semicontinuous.
  
\begin{corollary}
\label{cor1}
Let $V \in \F(\R^n)$ be defined on an open neighborhood $\O$ of $S$ and let assumptions~\ref{H2} and \ref{H3} be fulfilled. For $S^+_W = \{x\in S:\,  \sup_{\zeta \in \partial_{\bullet} V(x)}\langle \zeta,f(x)\rangle< 0 \}$, we have $\omega (x_0) \subseteq S\setminus S^+_W$. Furthermore, the $\omega$-limit set $\omega (x_0)$ is contained in a connected component of $S\setminus S^+_W$. In particular, the same result can also be obtained if $V\in C^1(\O,\R)$.
\end{corollary}
\begin{proof} 
The proof can be easily deduced by applying \cref{theorem:thm2} for 
\[
W(x):=\inf_{\zeta \in -\partial_{\bullet} V(x)} \langle \zeta,f(x)\rangle.
\] 
Moreover, since the $\omega$-limit set is connected, then it is contained in a connected component of $S\setminus S^+_W$.

Now, for $V \in C^1 (\O,\R)$, condition~\ref{H2} is fulfilled and the proof is completed by taking $W(x):= -\langle\nabla V(x), f(x) \rangle$ for all $x\in \O$.  
\end{proof}

\begin{remark}
\label{remark:rem2}
\cref{cor1} was also interpreted in \cite{DKV} for $V$ continuously differentiable and where \eqref{eq2} is viewed as a particular case of the standard differential inclusion problem. In our case, \eqref{eq2} is naturally covered when $A$ is the null operator. Moreover, always in the settings of \eqref{eq2}, it is important to mention that if $S$ is an invariant set, LaSalle's theorem (see \cref{theorem:thm0}) gives us the best location of the $\omega$-limit set. This latter case will be the aim of the next section. 
\end{remark}

\section{Location via invariant sets}
\label{sec4} 

Always in the spirit of the purpose of this paper, the aim of this section is to give further results about the location of the $\omega$-limit set. Indeed, if, in addition, the set $S$ is supposed to be invariant, we are willing to provide a stronger version  (than the one given in the standard LaSalle's invariance principle previously stated in \cref{theorem:thm0}) about the location of the $\omega$-limit set. To do so, we first propose a generalization of the LaSalle invariance principle for problem \eqref{E:eq1}. However, it will allow us to give outer estimates of the $\omega$-limit set in terms of the invariant set and in terms of nonsmooth Lyapunov-like functions.

\begin{theorem}
\label{theorem:thm2}
Let $S$ be a compact invariant set with respect to \eqref{E:eq1} contained in $\cl(\dom A)$. Suppose that $\inte(\co (\dom A)) \neq \varnothing$ and there exists a function $V \in \F(\R^n)$ such that  
\begin{equation}\label{eq10}
\forall x\in S,\quad \sup_{\zeta \in \partial_{\bullet} V(x)}\inf_{v \in A(x)}\langle \zeta,f(x) -v \rangle \leq 0,
\end{equation} 
where $\partial_{\bullet}$ stands for $\partial_P$ or $\partial_F$. For $\alpha \in \R,$ let $\mathfrak{M}_\alpha$ be the largest invariant set contained in $\bigcap_{\alpha < \beta} \cl \big([\alpha \leq V \leq \beta]_{\mid S}\big)$. Then, for each $x_0 \in S$,  
\begin{enumerate}
\item\label{inclusion}
there exists $\alpha \leq V(x_0)$ such that $\omega (x_0) \subseteq \mathfrak{M}_\alpha$;
\item\label{limit}
$\lim_{t \to +\infty} d\big(x(t;x_0); \bigcup_{\alpha \in \R}\mathfrak{M}_\alpha\big) = 0$.
\end{enumerate}
\end{theorem}
\begin{proof}
Let $x_0 \in S$ and denote by $x(t;x_0)$ the solution of \eqref{E:eq1}. Since the function $V$ is lower semicontinuous on the compact set $S$, it attains its minimum on $S$. Thus, there exists a point $\bar{x} \in S$ such that $V(x)$ is bounded below by $V(\bar{x})$ for every $x \in S$. In addition, according to \cite[Theorem~3.1]{Adly1} (or \cite[Corollary~3.14]{BTN}), condition~\eqref{eq10} implies that this function decreases along the solution of \eqref{E:eq1}. Therefore, $V(x(t;x_0)) \leq V(x(s;x_0))$ for all $0\leq s\leq t$, and $V(x(t;x_0))$ has a limit $\alpha$ as $t$ tends to $+\infty$.

For any $z \in \omega (x_0)$, there exists a sequence $(t_k)_{k\in \mathbb{N}}$ tending to $+\infty$ with $k\to +\infty$ such that 
\begin{equation}\label{eq12}
\lim_{k \to +\infty} x(t_k;x_0) = z. 
\end{equation}
Thus, since $V(x(t_k;x_0))$ is nonincreasing and $S$ is invariant, we have that, for all $j \geq 0$ and all $k \geq j$, 
\[
x(t_k;x_0) \in  [\alpha \leq V \leq V(x(t_j;x_0)]_{\mid S}.
\]
Using \eqref{eq12}, we get $z \in \cl \big([\alpha \leq V \leq V(x(t_k;x_0)]_{\mid S}\big)$ for $k \geq 0$. Therefore, for every $\beta > \alpha$, there exists $k \geq 0$ such that $x(t_k;x_0) \in [\alpha \leq V \leq \beta]_{\mid S}$ and we deduce that $z \in \cl \big([\alpha \leq V \leq \beta]_{\mid S}\big)$. Hence, 
\[
z \in  \bigcap_{\alpha < \beta} \cl \big([\alpha \leq V \leq \beta]_{\mid S}\big),
\]
which implies that  $\omega (x_0)$ is contained in $\bigcap_{\alpha < \beta} \cl \big([\alpha \leq V \leq \beta]_{\mid S}\big)$. In addition, since $\omega (x_0)$ is invariant, we derive that $\omega (x_0)$ is contained in the largest invariant set  $\mathfrak{M}_\alpha$ of $\bigcap_{\alpha < \beta} \cl \big([\alpha \leq V \leq \beta]_{\mid S}\big)$. Hence, for all $x_0 \in S$, $\omega (x_0) \subseteq \mathfrak{M}_\alpha \subseteq \bigcup_{\alpha \in \R} \mathfrak{M}_\alpha$. Since $\lim_{t \to +\infty} d(x(t;x_0),\omega (x_0)) = 0$, we deduce that \ref{limit} is verified.   
\end{proof}

\begin{remark}
\label{remark:rem3}
In the case where the condition $\inte(\co (\dom A)) \neq \varnothing$ or the local boundedness of $A$ is not assumed, we need to consider the horizon subdifferential of $V$. In this case, condition~\eqref{eq10} is replaced by
\[
\forall x\in S,\quad \sup_{\zeta \in \partial_{P} V(x) \cup \partial_{\infty} V(x)}\inf_{v \in A(x)}\langle \zeta,f(x) -v \rangle \leq 0.
\]
For further information, we invite the readers to check \cite{Adly1} and/or \cite{Adly3}.
\end{remark}

\begin{remark}
\label{remark:rem4}
Note that \cref{theorem:thm2}\ref{inclusion} can also be found in \cite{Haddad} in the context of ordinary differential equations. Moreover, in the setting of \eqref{eq2}, \cref{theorem:thm0} can be easily deduced from \cref{theorem:thm2} (see \cite{Haddad}). Indeed, \cref{theorem:thm2}\ref{inclusion} and the continuity of $V$ insure the existence, for $x_0 \in S$, of an $\alpha \in \R$ such that $\omega (x_0) \subseteq \mathfrak{M}_\alpha$, where $\mathfrak{M}_\alpha$ is the largest invariant set contained in $[V = \alpha]_{\mid S}$ which proves that $\omega (x_0) \subseteq [V = \alpha]_{\mid S}$. Furthermore, it is easy to verify that the invariant set $\mathfrak{M}_\alpha$ is contained in the largest invariant set contained in $E$ in order to deduce the result of \cref{theorem:thm0}. Therefore, this allows us to give a more accurate location of the $\omega$-limit set, compared to the one proposed by the standard LaSalle's theorem, since 
\[
\omega (x_0) \subseteq \mathfrak{M}_\alpha \subseteq M \subseteq E \subseteq S.
\] 
Finally, in the case where the function $V$ is lower semicontinuous, it suffices to replace condition \eqref{eq10} by 
\begin{equation*}
\forall x\in S,\quad \sup_{\zeta \in \partial_{\bullet} V(x)}\langle \zeta,f(x) \rangle \leq 0.
\end{equation*}
\end{remark}

In the remainder of this section, we propose two corollaries of \cref{theorem:thm2} for the particular but important case where $A \equiv \partial \varphi$ with $\varphi \in \F(\R^n)$ and convex.  Note that in this case, the operator $\partial \varphi$ is not necessarily locally bounded on $\dom V$ (in particular on $S$), thus the condition $\inte(\co (\dom A)) \neq \varnothing$ is not satisfied. Therefore, according to \cref{remark:rem3}, the use of the horizon subdifferential of $V$ is required in order to fill the gap. In the first corollary, we consider the case where $V\in \F(\R^n)$, while the second deals with the case where $V \in C^1$.
  
\begin{corollary}
\label{corollary:cor4} 
Let $S$ be a compact invariant set contained in $\cl (\dom \varphi)$. Suppose that there exists a function $V \in \F(\R^n)$ such that, for all $x\in S$,
\begin{equation}\label{eq13}
\sup_{\zeta \in \partial_{P} V(x) \cup \partial_{\infty} V(x)}\inf_{v \in \partial \varphi(x)}\langle \zeta,f(x) -v \rangle \leq 0. 
\end{equation}
Then the conclusions of \cref{theorem:thm2} hold.
\end{corollary}
\begin{proof}
The result is a direct consequence of \cref{theorem:thm2} and \cite[Theorem~3.3]{Adly3}.	
\end{proof}

\begin{corollary}
\label{corollary:cor5} 
Under the assumptions of \cref{corollary:cor4}, suppose that there exists a continuously differentiable function $V$ such that, for all $x\in S$,  
\begin{equation}\label{eq14}
\max_{v \in -\partial  \varphi(x)}\langle \nabla V(x),f(x)+v \rangle \leq 0. 
\end{equation} 
Let $\mathfrak{E}_S := \{x\in S:\, \langle \nabla V(x),f(x)+v \rangle = 0 \text{~for some~} v \in -\partial \varphi(x)\}$ and let $\mathfrak{M}_S$ be the largest invariant set of $\mathfrak{E}_S$. Then, for each $x_0 \in S$, 
\begin{equation}\label{eq15}
\lim_{t \to +\infty} d(x(t;x_0),\mathfrak{M}_S) = 0.
\end{equation}
\end{corollary}
\begin{proof}
Using \eqref{eq14}, we have 
\begin{align*}
\dot{V}(x(t;x_0)) & = \langle\nabla V(x(t;x_0)),\dot{x}(t;x_0) \rangle \\ 
& \leq \max_{v \in -\partial \varphi(x(t;x_0))}\langle \nabla V(x(t;x_0)),f(x(t;x_0))+v \rangle \\
& \leq 0 \text{~~~~a.e.~~} t\in [0,+\infty),
\end{align*}	
which means that $V$ is decreasing along the trajectories. Now, following the same arguments discussed in \cref{remark:rem4}, we can prove the existence of an $\alpha \in \R$ such that $\omega (x_0) \subseteq \mathfrak{M}_\alpha$, where $\mathfrak{M}_\alpha$ is the largest invariant set contained in $[V =\alpha]_{\mid S}$. Moreover, 
\[
\frac{dV(x(t))}{dt}\mid_{t = 0} = \langle\nabla V(x(t)),\dot{x}(t) \rangle\mid_{t=0} = 0, 
\]  
which proves that $\mathfrak{M}_\alpha$ is contained in $\mathfrak{M}_S$. Therefore, since $x(t;x_0)$ approaches $\omega (x_0) \subseteq \mathfrak{M}_\alpha \subseteq \mathfrak{M}_S$, \eqref{eq15} is deduced.    
\end{proof}

\begin{remark}
\label{remark:rem5} 
Some comments on the two corollaries above are in order.
\begin{enumerate}
\item 
\cref{corollary:cor4,corollary:cor5} can be considered, respectively, as a generalization and a refinement of LaSalle's invariance theorem given in \cite{Qi}. More precisely, \cref{corollary:cor5} provides a better location of the $\omega$-limit set in terms of the invariant sets. In \cite{Qi}, the author shows that the $\omega$-limit set is also contained in the largest invariant set $\mathfrak{M}$ of the set $\mathfrak{E} := \{x\in S:\, \langle f(x), \nabla V(x)\rangle + \varphi(x) - \varphi (x - \nabla V(x)) = 0\}$, which is contained in $\mathfrak{M}_S$. Meanwhile, \cref{corollary:cor5} gives the following location of the $\omega$-limit set: 
\[
\omega(x_0) \subseteq \mathfrak{M}_\alpha \subseteq  \mathfrak{M} \subseteq \mathfrak{M}_S \subseteq S.
\]
\item 
Let $\mathfrak{P}$ be the largest invariant subset of the invariant set $\{y : V(y) \leq V(x_0)\}$ containing $\omega (x_0)$ (see \cite{Adlybook}). Then, according to the proofs of \cref{theorem:thm2} and \cref{corollary:cor5}, we can easily provide a sharper location of $\omega(x_0)$ given as follows:
\[
\omega (x_0) \subseteq \mathfrak{P} \cap \mathfrak{M}_\alpha. 
\]    
\end{enumerate}
\end{remark}

We conclude this paper by giving an example of  an important second-order gradient-like dissipative dynamical system with Hessian-driven damping, called the \emph{dynamical inertial Newton-like} (DIN) system by Alvarez et al.~\cite{AABR02}\footnote{We would like to thank one of the referees for   pointing our attention to this reference.}. These authors observed  that this system can be equivalently written as a  Lipschitz perturbation of  a gradient system, thus reaching the setting of our  paper.  We refer the readers to \cite{HJ15} for more concrete examples concerning dynamical systems and optimization.
\begin{example}[DIN system revisited]
\label{example:ex2} 
Let $\Phi : \R^n \to \R$ be a twice differentiable function, bounded from below, whose Hessian  $\nabla^2 \Phi$ is Lipschitz continuous on bounded subsets of $\R^n$. We consider the second-order dynamical system
\begin{equation}
\label{eq16}
\ddot{x}(t) + \alpha \dot{x}(t) +\nabla \Phi (x(t))  + \beta \nabla^2 \Phi (x(t)) \dot{x}(t) = 0,
\end{equation}
where $\alpha$ and $\beta$ are strictly positive parameters. This system plays an important role in mechanics, optimization, and control theory. In \cite{AABR02}, the authors established the existence and uniqueness of the solution of \eqref{eq16} satisfying the initial conditions $x(0) =x_0$ and $\dot{x}(0)=\dot{x}_0$. Using the {\L}ojasiewicz inequality, they showed that each trajectory of \eqref{eq16} converges to a critical point of $\Phi$ when $\Phi$ is analytic. They also proved the convergence of the trajectory to a minimizer of $\Phi$ when $\Phi$ is convex and $\argmin \Phi$ is nonempty. 

The aim of this example is to apply our theoretical results to the system given by \eqref{eq16} to obtain a convergence result and to provide a localization of $\omega(x_0)$. First, \eqref{eq16} can be equivalently written, in the phase space $\R^n \times \R^n,$ as a first-order Cauchy--Lipschitz problem of the form
\begin{equation}
\label{eq17}
\dot{y}(t) =  f(y(t)), 	
\end{equation}
where $y = \begin{bmatrix} x \\ \dot{x} \end{bmatrix}$ and $f(y_1,y_2) = \begin{bmatrix} y_{2} \\ -\alpha y_2 - \nabla\Phi (y_1)  - \beta \nabla^2 \Phi (y_1)y_2 \end{bmatrix}$. 
   
Let $\mathcal{S} =\{(y_1,y_2)\in \R^n\times \R^n:\, f(y_1,y_2) =0\}$ be the set of \emph{equilibrium points} of \eqref{eq17}. Then $\mathcal{S} = \mathcal{N} \times \{0\}$, where $\mathcal{N} = \{x \in \R^n:\, \nabla \Phi (x) =0 \}$ is the set of \emph{critical points} of $\Phi$. Next, we denote by $y(t;y_0)$ the solution of \eqref{eq17} starting at $y_0 = y(0) \in \R^n \times \R^n$. For all $y_0 =(x_0,0) \in \mathcal{S}$, we have that, for all  $t \geq 0$, $y(t;y_0) = y_0 \in \mathcal{S}$. So, the set $\mathcal{S} $ is invariant with respect to \eqref{eq17}. Moreover, consider the function $V$ defined by
\[
V(y(t)) = (\alpha \beta +1) \Phi (x(t)) + \frac{1}{2} \| \dot{x}(t) + \beta \nabla \Phi(x(t)) \|^2. 
\]
We see that the derivative of $V$ along the solution of \eqref{eq17} satisfies
\begin{eqnarray*}
\dot{V}(y(t)) & = & \Big\langle \nabla V(x(t)),\dot{x}(t) \Big\rangle  \\
& = &  -\alpha \| \dot{x}(t)\|^2 - \beta \| \nabla \Phi (x(t)) \| ^2 \leq 0.
\end{eqnarray*}
It follows that the set $\mathcal{S}$ coincides with the set of solutions where the derivative of $V$ along the solution vanishes. 

By applying \cref{corollary:cor5} and the consequences discussed in \cref{remark:rem4,remark:rem5}, we derive that 
\[
\lim_{t \to +\infty} d(y(t),\mathcal{S}) = 0,
\]
which yields
\begin{equation}
\label{eq18}
\lim_{t \to +\infty} d(x(t),\mathcal{N}) = 0 \text{~~and~~} \lim_{t \to +\infty} \dot{x} (t) = 0. 
\end{equation}
This proves that the set $\mathcal{N}$ of critical points of $\Phi$ is attractive. Moreover, by \cref{theorem:thm2} and \cref{remark:rem4}, there exists $\alpha \in \R$ such that the $\omega$-limit set $\omega(x_0)$ is contained in $\mathfrak{M}_\alpha \subseteq \mathcal{N}$, where $\mathfrak{M}_\alpha$ is the largest invariant set contained in $[V =\alpha]$.
\end{example}

\section{Conclusion}

We have extended the work done by Dontchev, Krastanov, and Veliov \cite{DKV} to differential inclusions governed by a maximally monotone operator by justifying it mathematically using techniques from nonsmooth and variational analysis. In addition, we have generalized the LaSalle's invariance principle under nonsmooth data. Compared to the standard invariance principle, our result gives a more accurate location of the $\omega$-limit set. A concrete example is discussed to illustrate the application of our results to the situations useful for optimization. Finally, we hope that it will succeed in stimulating enough interest in the community for applying our results to some new problems arising in possible applications.

\section*{Acknowledgment}
 
The authors would like to thank the anonymous referees for constructive comments and suggestions that helped to improve the manuscript.

\end{document}